\documentclass[12pt]{amsart}
\usepackage{amssymb, enumerate, hyperref}
\newcommand{\excise}[1]{}

\newtheorem{thm}{Theorem}[section]
\newtheorem{lemma}[thm]{Lemma}

\theoremstyle{definition}

\newcommand\<{\langle}
\newcommand\1{\mathbf{1}}
\newcommand\CC{{\mathbb C}}

\newcommand\RR{{\mathbb R}}
\newcommand\ZZ{{\mathbb Z}}
\newcommand\kk{\Bbbk}
\newcommand\del{\partial}
\renewcommand\>{\rangle}
\renewcommand\AA{{\mathbb A}}

\newcommand\ini{\operatorname{in}}
\newcommand\Orb{{\rm Orb}}
\newcommand\rank{\operatorname{rank}}
\newcommand\ann{\operatorname{ann}}
\newcommand\Var{\operatorname{Var}}
\newcommand\gr{\mathrm{gr}}

\begin{document}

\mbox{}\vspace{-4ex}
\title{Systems of parameters and holonomicity of $A$-hypergeometric systems}

\author{Christine Berkesch}
\address{Department of Mathematics, Duke University, Box 90320, Durham, NC 27708}
\email{cberkesc@math.duke.edu}

\author{Stephen Griffeth}
\address{Instituto de Mathem\'atica y F\'isica, Universidad de Talca,
Camino Lircay, Talca, Chile}
\email{sgriffeth@inst-mat.utalca.cl}

\author{Ezra Miller}
\address{Department of Mathematics, Duke University, Box 90320, Durham, NC 27708}
\email{ezra@math.duke.edu}

\thanks{EM had support from NSF grant DMS-1001437.  SG acknowledges
the financial support of Fondecyt Proyecto Regular 1110072}

\date{22 January 2013}

\makeatletter\@namedef{subjclassname@2010}{\textup{2010} Mathematics
Subject Classification}\makeatother

\begin{abstract}
The main result is an elementary proof of holonomicity for
$A$-hypergeometric systems, with no requirements on the behavior of
their singularities, originally due to Adolphson \cite{adolphson}
after the regular singular case by Gelfand and
Gelfand~\cite{gelfand--gelfand}.  Our method yields a direct de~novo
proof that $A$-hypergeometric systems form holonomic families over
their parameter spaces, as shown by Matusevich, Miller, and
Walther~\cite{MMW}.
\end{abstract}

\maketitle

\section*{Introduction}

An $A$-hypergeometric system is the $D$-module counterpart of a toric
ideal.  Solutions to $A$-hypergeometric systems are functions, with a
fixed infinitesimal homogeneity, on an affine toric variety.
The solution space of an $A$-hypergeometric system behaves well in
part because the system is holonomic, which in particular implies that
the vector space of germs of analytic solutions at any nonsingular
point has finite dimension.

This note provides an elementary proof of holonomicity for arbitrary
$A$-hypergeometric systems, relying only on the statement that a
module over the Weyl algebra in $n$ variables is holonomic precisely
when its characteristic variety has dimension at most~$n$~(see~\cite{gabber} or~\cite[Theorem~1.12]{borel}), along with standard facts about transversality of
subvarieties and about Krull dimension.  In particular, our proof requires no assumption
about the singularities of the $A$-hypergeometric system;
equivalently, the associated toric ideal need not be standard graded.
Holonomicity was proved in the regular singular case by Gelfand and
Gelfand~\cite{gelfand--gelfand}, and later by
Adolphson~\cite[\S3]{adolphson} regardless of the behavior of the
singularities of the system.  Adolphson's proof relies on careful
algebraic analysis of the coordinate rings of a collection of
varieties whose union is the characteristic variety of the system.
Another proof of the holonomicity of an $A$-hypergeometric system, by
Schulze and Walther~\cite{schulze-walther-duke}, yields a more general
result: for a weight vector~$L$ from a large family of possibilities,
the $L$-characteristic variety for the $L$-filtration is a union of
conormal varieties and hence has dimension~$n$; holonomicity follows
when $L = (0,\dots,0,1,\dots,1)$ induces the order filtration on the
Weyl algebra.  The $L$-filtration method uses an explicit
combinatorial interpretation of initial ideals of toric ideals, which
requires a series of technical lemmas.

Holonomicity of $A$-hypergeometric systems forms part of the statement
and proof, by Matusevich, Miller, and Walther~\cite{MMW}, that
$A$-hypergeometric systems determine holonomic families over their
parameter spaces.  The new proof of that statement here serves as a
model suitable for generalization to hypergeometric systems for
reductive groups, in the sense of Kapranov \cite{kapranov98}.

The main step (Theorem~\ref{t:main}) in our proof is an easy geometric
argument showing that the Euler operators corresponding to the rows of
an integer matrix~$A$ form part of a system of parameters on the
product $\kk^n \times X_{\hspace{-.1ex}A}$, where $\kk$ is any
algebraically closed field and $X_{\hspace{-.1ex}A}$ is the toric
variety over~$\kk$ determined by~$A$.  This observation leads quickly
in Section~\ref{s:holonomic} to the conclusion that the characteristic
variety of the associated $A$-hypergeometric system has dimension at
most~$n$, and hence that the system is holonomic.  Since the algebraic
part of the proof holds when the entries of~$\beta$ are considered as
independent variables that commute with all other variables, the
desired stronger consequence is immediate: the $A$-hypergeometric
system forms a holonomic family over its parameter space
(Theorem~\ref{t:holonomic}).

\section{Systems of parameters via transversality}\label{s:systems}

Fix a field~$\kk$.  Let $x = x_1,\dots,x_n$ and $\xi = \xi_1,\dots,\xi_n$
be sets of coordinates on $\kk^n$ and let $x\xi$ denote the column
vector with entries $x_1\xi_1, \ldots, x_n\xi_n$.  Given a rectangular
matrix $L$ with $n$ columns, write $Lx\xi$ for the vector of bilinear
forms given by multiplying $L$~times~$x\xi$.

\begin{lemma}\label{l:transverse}
Let $\kk^{2n} = \kk^n_x \times \kk^n_\xi$ have coordinates $(x,\xi)$
and let $X \subseteq \kk^n_\xi$ be a subvariety.  If $L$ is an $\ell
\times n$ matrix with entries in~$\kk$, then the variety $\Var(L
x\xi)$ of $L x\xi$ in~$\kk^{2n}$ is transverse to $\kk^n \times X$ at
any smooth point of\/~$\kk^n \times X$ whose $\xi$-coordinates are all
nonzero.
\end{lemma}
\begin{proof}
It suffices to prove the statement after passing to the algebraic
closure of~$\kk$, so assume $\kk$ is algebraically closed.  Let
$(p,q)$ be a smooth point of $\kk^n \times X$ that lies in
$\Var(Lx\xi)$ and has all coordinates of~$q$ nonzero.  The tangent
space to $\kk^n \times X$ at~$(p,q)$ contains $\kk^n \times \{0 \}$.
The tangent space~$T_{(p,q)}$ to $\Var(Lx\xi)$ is the kernel of the
$\ell \times 2n$ matrix $[L(q)\ L(p)]$, where $L(p)$
(respectively,~$L(q)$) is the $\ell \times n$ matrix that results
after multiplying each column of~$L$ by the corresponding coordinate
of~$p$ (respectively,~$q$).  Since the $q$ coordinates are all
non-zero, $T_{(p,q)}$ projects surjectively onto the last $n$
coordinates; indeed, if $\eta \in \kk^n_\xi$ is given, then taking
$y_i = -p_i\eta_i/q_i$ yields $y \in \kk^n_x$ with $L(q)y + L(p)\eta =
0$.
Thus the tangent spaces at $(p,q)$ sum to the ambient space, so the
intersection is transverse.
\end{proof}

The next result applies the lemma to an affine toric variety~$X$.  A
fixed $d \times n$ integer matrix $A = [a_1\ a_2\ \cdots\ a_{n-1}\
a_n]$ defines an action of the algebraic torus $T = (\kk^*)^d$ on
$\kk^n_\xi$~by
$$%
  t \cdot \xi = (t^{a_1}\xi_1, \ldots, t^{a_n}\xi_n).
$$
The orbit $\Orb(A)$ of the point $\1 = (1, \ldots, 1) \in \kk^n$ is
the image of an algebraic map $T \to \kk^n$ that sends $t \to t \cdot
\1$.  The closure of $\Orb(A)$ in~$\kk^n$ is the affine toric variety
$X_{\hspace{-.1ex}A} = \Var(I_A)$ cut out by the \emph{toric ideal}
$$%
  I_A
  =
  \<\xi^u - \xi^v \mid Au = Av\>
  \subseteq
  \kk[\xi],
$$ 
of~$A$ in the polynomial ring $\kk[\xi] = \kk[\xi_1,\dots,\xi_n]$.
The $T$-action induces an \emph{$A$-grading} on~$\kk[\xi]$ via
$\deg(\xi_i) = a_i$, and the semigroup ring $S_A = \kk[\xi]/I_A$ is
$A$-graded \cite[Chapters~7--8]{cca}.

For any face~$\tau$ of the real cone $\RR_{\geq 0} A$ generated by the
columns of~$A$, write $\tau \preceq A$ and let $\1^\tau \in \{0,1\}^n
\subset \kk^n$ be the vector with nonzero entry $\1^\tau_i = 1$
precisely when $A$ has a nonzero column $a_i \in \tau$.  The variety
$X_{\hspace{-.1ex}A}$ decomposes as a finite disjoint union
$X_{\hspace{-.1ex}A} = \bigsqcup_{\tau\preceq A} \Orb(\tau)$ of
orbits, where $\Orb(\tau) = T \cdot \1^\tau$.  Each orbit has
dimension $\dim\Orb(\tau) = \rank(A_\tau)$, where $A_\tau$ is the
submatrix of~$A$ consisting of those columns lying in~$\tau$, and
$\dim X_{\hspace{-.1ex}A} = \rank(A)$.
 
\begin{thm}\label{t:main}
The ring $\kk[x,\xi]/\big(I_A+\<Ax\xi\>\big)$ has Krull dimension~$n$.
In particular, if $A$ has rank~$d$ then the forms $Ax\xi$ are part of a system of parameters for $\kk[x] \otimes_\kk S_A$.
\end{thm}
\begin{proof}
Let $\kk^\tau \subseteq \kk^n$ be the subspace consisting of vectors
with~$0$ in coordinate~$i$ if $a_i \notin \tau$, and let $|\tau|$ be
its dimension.  Since $\kk[x,\xi]/I_A = \kk[x] \otimes_\kk S_A$ has
dimension $n + \rank(A)$ and the number of $\kk$-linearly independent
generators of $\<Ax\xi\>$ is at most $\rank(A)$, the Krull dimension
in question is at least~$n$.  Hence it suffices to prove that
$\big(\kk^n \times \Orb(\tau)\big) \cap \Var(A x\xi) \subseteq \kk^n
\times \kk^\tau$ has dimension at most~$n$.
Let $x_\tau$ and $\xi_\tau$ denote the subsets corresponding to~$\tau$
in the variable sets $x$ and~$\xi$, respectively.  The projection of
the intersection onto the subspace $\kk^\tau \times \kk^\tau$ has
image contained in
$$%
  \big(\kk^\tau \times \Orb(\tau)\big) \cap \Var(A_\tau x_\tau\xi_\tau)
  \subseteq \kk^\tau \times \kk^\tau.
$$ 
It therefore suffices to show that the dimension of this latter
intersection is at most $|\tau|$.  By Lemma~\ref{l:transverse}, the
intersection is transverse in $\kk^\tau \times \kk^\tau$.  But the
dimension of $\Orb(\tau)$ is the codimension of $\Var(A_\tau x_\tau
\xi_\tau)$ in $\kk^\tau \times \kk^\tau$, which completes the proof. 
\end{proof}

\section{Hypergeometric holonomicity}\label{s:holonomic}

In this section, the matrix $A$ is a $d\times n$ integer matrix of full rank $d$.
Let 
$$%
D = \CC\<x,\del \,\big|\, [\del_i,x_j] = \delta_{ij} \text{ and }
    [x_i,x_j] = 0 = [\del_i,\del_j]\>
$$
denote the Weyl algebra over the complex numbers~$\CC$, where $x =
x_1,\dots,x_n$ and $\del_i$ corresponds to $\frac{\del}{\del x_i}$.
This is the ring of $\CC$-linear differential operators on $\CC[x]$.

For $\beta\in\CC^d$, the \emph{$A$-hypergeometric system} with
parameter~$\beta$ is the left $D$-module
$$%
  M_A(\beta) 
  =
  D/D\cdot(I_A^\del,\{E_i-\beta_i\}_{i=1}^d),
$$
where $I_A^\del = \<\del^u - \del^v \mid Au = Av\> \subseteq
\CC[\del]$ is the toric ideal associated to~$A$ and
$$%
  E_i-\beta_i =\sum_{j=1}^na_{ij}x_j\del_j - \beta_i
$$
are \emph{Euler operators} associated to~$A$.

The \emph{order filtration}~$F$ filters~$D$ by order of differential
operators.  The \emph{symbol} of an operator~$P$ is its image $\ini(P)
\in \gr^F D$.  Writing $\xi_i = \ini(\del_i)$, this means $\gr^F D$ is
the commutative polynomial ring $\CC[x,\xi]$.  The
\emph{characteristic variety} of a left $D$-module~$M$ is the variety
in $\AA^{2n}$ of the associated graded ideal $\gr^F\ann(M)$ of the
annihilator of~$M$.  A nonzero $D$-module is \emph{holonomic} if its
characteristic variety has dimension~$n$; this is equivalent to
requiring that the dimension be at most~$n$; see~\cite{gabber}
or~\cite[Theorem~1.12]{borel}.
The \emph{rank} of a holonomic $D$-module $M$ is the (always finite)
dimension of $\CC(x)\otimes_{\CC[x]}M$ as a vector space
over~$\CC(x)$; this number equals
the dimension of the vector space of germs of analytic solutions
of~$M$ at any nonsingular point
in~$\CC^n$~\mbox{\cite[Theorem~1.4.9]{SST}}.

Viewing the $A$-hypergeometric system $M_A(\beta)$ as having a varying
parameter $\beta \in \nolinebreak\CC^d$, the rank of $M_A(\beta)$ is
upper semicontinuous as a function of~$\beta$ \cite[Theorem~2.6]{MMW}.
This follows by viewing $M_A(\beta)$ as a \emph{holonomic family}
\cite[Definition~2.1]{MMW} para\-metrized by $\beta\in\CC^d$.  By
definition, this means not only that $M_A(\beta)$ is holonomic for
each~$\beta$, but also that it satisfies a coherence condition
over~$\CC^d$, namely: after replacing $\beta$ with variables $b =
b_1,\dots,b_d$, the module $\CC(x)\otimes_{\CC[x]} M_A(b)$ is finitely
generated over $\CC(x)[b]$.  (The definition of holonomic family in \cite{MMW} allows
sheaves of $D$-modules over arbitrary complex base schemes, but that
generality is not needed here.)

The derivation of the holonomic family property for $M_A(b)$ from the
holonomicity of the $A$-hypergeometric system is more or less the same
as in \cite[Theorem~7.5]{MMW}, which was phrased in the generality of
Euler--Koszul homology of toric modules.  The brief deduction here
isolates the steps necessary for $A$-hypergeometric systems; its
brevity stems from the special status of affine semigroup rings among
all toric modules \cite[Definition~4.5]{MMW}.

\begin{thm}\label{t:holonomic}
The module $M_A(b)$ forms a holonomic family over $\CC^d$ with
coordinates~$b$.  In more detail, as a $D[b]$-module the parametric
$A$-hypergeometric system $M_A(b)$ satisfies:
\begin{enumerate}[\quad1.]
\item\label{t:holonomic.1}%
  the fiber $M_A(\beta) = M_A(b)\otimes_{\CC[b]}\CC[b]/\<b-\beta\>$ is
  holonomic for all $\beta$; and
\item\label{t:holonomic.2}%
  the module $\CC(x)\otimes_{\CC[x]} M_A(b)$ is finitely generated
  over $\CC(x)[b]$.
\end{enumerate}
\end{thm}
\begin{proof}
Since $R = \CC[x,\xi]/\<\ini(I_A), Ax\xi\>$ surjects onto $\gr^F
M_A(\beta)$,
it is enough to show that the ring~$R$ has dimension $n$.  If
$M_A(\beta)$ is standard $\ZZ$-graded (equivalently, the rowspan
of~$A$ over the rational numbers contains the row $[1\ 1\ \cdots\ 1\
1]$ of length~$n$), then $\ini(I_A) = I_A\subseteq\CC[\xi]$, and the
result follows from Theorem~\ref{t:main}.

When $M_A(\beta)$ is not standard $\ZZ$-graded, let $\hat{A}$ be the
$(d+1) \times (n+1)$ matrix obtained by adding a row of $1$'s across
the top of~$A$ and then adding as the leftmost column $(1,0,\dots,0)$.
If $\xi_0$ denotes a new variable corresponding to the leftmost column
of~$\hat{A}$, and $\hat{\xi} = \{\xi_0\}\cup\xi$, then
$\CC[\xi]/\ini(I_A) \cong \CC[\hat{\xi}]/\<I_{\hat{A}},\xi_0\>$.  In
particular,
$$%
  \frac{\CC[\hat{x},\xi]}{\<\ini(I_A), Ax\xi\>}
  \cong 
  \frac{\CC[\hat x,\hat\xi]}{\<I_{\hat A},\xi_0,\hat A\hat x\hat\xi\>}, 
$$
where $\hat x = \{x_0\} \cup x$.  Since $\<I_{\hat A},\xi_0\>$ is
$\hat A$-graded and $\hat A$ has a row $[1\ 1\ \cdots\ 1\ 1]$, we have
reduced to the case where $M_A(\beta)$ is $\ZZ$-graded, completing
part~\ref{t:holonomic.1}.

With $R$ as in part~\ref{t:holonomic.1}, the ring $R[b]$ surjects
onto~$\gr^F M_A(b)$, so it suffices for part~\ref{t:holonomic.2} to
show that $R[b]$ becomes finitely generated over $\CC(x)[b]$ upon
inverting all nonzero polynomials in~$x$.  Since the ideal
$\<\ini(I_A), Ax\xi\>$ has no generators involving $b$~variables, it
suffices to show that $R(x)$ itself has finite dimension
over~$\CC(x)$.  The desired result from the statement proved for
part~\ref{t:holonomic.1}: any scheme of dimension~$n$ has finite degree over~$\CC^n_x$.
\end{proof}


\begin{thebibliography}{BGK$^K$87}

\bibitem[Ado94]{adolphson}
Alan~Adolphson, 
\emph{Hypergeometric functions and rings generated by monomials}, 
Duke Math. J. \textbf{73} (1994), 269--290.

\bibitem[BGK$^+$87]{borel}
A.~Borel, P.-P.~Grivel, B.~Kaup, A.~Haefliger, B.~Malgrange, and F.~Ehlers, 
\emph{Algebraic D-modules}, 
Perspectives in Mathematics \textbf{2}, Academic Press, Inc., Boston, MA, 1987.

\bibitem[Gab81]{gabber}
Ofer~Gabber, 
\emph{The integrability of the characteristic variety}, 
Amer. J. Math. \textbf{103} (1981), no. 3, 445--468. 

\bibitem[GG86]{gelfand--gelfand}
I. M.~Gelfand and S. I.~Gelfand, \emph{Generalized hypergeometric
  equations}, (Russian) Dokl. Akad. Nauk SSSR \textbf{288} (1986),
  no.~2, 279--283.  English Translation:  Soviet Math.\ Dokl.\
  \textbf{33} (1986), no.~3, 643--646.

\bibitem[GGZ87]{GGZ}
I.~M.~Gelfand, M.~I.~Graev, and A.~V.~Zelevinsky, \emph{Holonomic
  systems of equations and series of hypergeometric type},
  Dokl.\ Akad.\ Nauk SSSR \textbf{295} (1987), no.~1, 14--19.

\bibitem[GKZ90]{GKZ-int}
I.~M.~Gelfand, M.~Kapranov, and A.~V.~Zelevinsky, \emph{Generalized
  Euler integrals and {$A$}-hypergeometric functions},
  Adv. Math. \textbf{84} (1990), no. 2, 255--271.

\bibitem[GKZ93]{GKZ-hyp}
I.~M.~Gelfand, A.~V.~Zelevinsky, and M.~M.~Kapranov,
  \emph{Hypergeometric functions and toric varieties},
  Funktsional.\ Anal.\ i Prilozhen.\ \textbf{23} (1989), no.~2,
  12--26.  Correction in ibid, \textbf{27} (1993), no.~4,~91.

\bibitem[GKZ94]{GKZ-discrim}
I.~M.~Gelfand, M.~Kapranov, and A.~V.~Zelevinsky, 
\emph{Discriminants, resultants and multidimensional determinants}, 
Mathematics: Theory \& Applications. Birkh\"auser Boston, Inc., Boston, MA, 1994. 

\bibitem[Kap98]{kapranov98}
Mikhail Kapranov, \emph{Hypergeometric functions on reductive groups},
  Integrable systems and algebraic geometry (Kobe/Kyoto, 1997), World
  Sci. Publishing, River Edge, NJ, 1998, pp.~236--281.

\bibitem[MMW05]{MMW}
Laura~Felicia~Matusevich, Ezra~Miller, and Uli~Walther, 
\emph{Homological methods for hypergeometric families}, 
J. Amer. Math. Soc. \textbf{18} (2005), no.~4, 919--941.

\bibitem[MS]{cca}
Ezra~Miller and Bernd~Sturmfels, 
\emph{Combinatorial commutative algebra}, 
Graduate Texts in Mathematics, \textbf{227}. Springer-Verlag, New York, 2005.

\bibitem[SST00]{SST}
Mutsumi~Saito, Bernd~Sturmfels, and Nobuki~Takayama, 
\emph{Gr\"obner {D}eformations of {H}ypergeometric {D}ifferential {E}quations}, 
Springer--Verlag, Berlin, 2000.

\bibitem[SW08]{schulze-walther-duke}
Mathias~Schulze and Uli~Walther, 
\emph{Irregularity of hypergeometric systems via slopes along coordinate subspaces}, 
Duke Math. J. \textbf{142} (2008),  no. 3, 465--509.

\end{thebibliography}
\end{document}